\pdfoutput=1
\documentclass[english]{jnsao}
\usepackage[utf8]{inputenc}
\usepackage[nameinlink,capitalize]{cleveref}

\usepackage{empheq}

\newcommand{\RR}{\ensuremath{\mathbb R}}
\newcommand{\RP}{\ensuremath{\mathbb{R}_+}}
\newcommand{\scal}[2]{\left\langle{#1},{#2}  \right\rangle}
\newcommand{\menge}[2]{\big\{{#1}~\big |~{#2}\big\}}
\newcommand{\mmenge}[2]{\bigg\{{#1}~\bigg |~{#2}\bigg\}}

\definecolor{myblue}{rgb}{.8, .8, 1}
  \newcommand*\mybluebox[1]{%
    \colorbox{structure!7}{\hspace{1em}#1\hspace{1em}}}

\title{Minimal angle spread in the probability simplex\\ with respect to the uniform distribution}
\shorttitle{Minimal angle spread in the probability simplex}
\author{
Heinz H. Bauschke\thanks{%
Mathematics, University
of British Columbia,
Kelowna, B.C. V1V~1V7, Canada.
\email{heinz.bauschke@ubc.ca}, \orcid{0000-0002-4155-9930}.}
\and
Peter A.\,V. DiBerardino\thanks{%
Psychology, University of Waterloo, Waterloo, ON N2L~3G1, Canada.
\email{pavdiber@uwaterloo.ca}, \orcid{0000-0001-7607-9278}.}
}
\shortauthor{Bauschke, DiBerardino}

\acknowledgements{
The authors thank the editors and two anonymous reviewers for their constructive comments. These led us to explore possible connections to
dihedral angles (see the paragraph following \cref{t:main}),
to include \cref{sec:app}, and to point out a possible connection to strong convexity of
integral functionals at the end of \cref{sec:intro}.
HHB is supported by the Natural Sciences and
Engineering Research Council of Canada.
PAVD thanks Dr.\ Britt Anderson for his constructive comments and support.
}

\manuscriptcopyright{{\copyright} the authors}
\manuscriptlicense{CC-BY-SA 4.0}

\manuscriptsubmitted{2021-05-19}
\manuscriptaccepted{2022-04-22}
\manuscriptvolume{3}
\manuscriptnumber{7492}
\manuscriptyear{2022}
\manuscriptdoi{10.46298/jnsao-2022-7492}

\manuscripteprinttype{arXiv}
\manuscripteprint{2105.08653}
\begin{document}
%%%%%%%%%%%%%%%%%%%%%%%%%%%%%%%%%%%%%%%%%%%%%%%

\maketitle

\begin{abstract}
We compute the minimal angle spread with respect to
the uniform distribution in the probability
simplex.
The resulting optimization problem is analytically solved.
The formula provided shows that the minimal angle spread approaches zero as the
dimension tends to infinity.
We also discuss an application in cognitive science.
\end{abstract}

\section{Introduction}
\label{sec:intro}

Throughout this paper, we assume that
\begin{empheq}[box=\mybluebox]{equation}
%\begin{equatkion}
    \text{$X := \RR^n$
    with inner product
    $\scal{\cdot}{\cdot}\colon X\times X\to\RR$, }
%    \end{equation}
\end{empheq}
and induced Euclidean norm $\|\cdot\|$.
We also define the probability simplex by
\begin{empheq}[box=\mybluebox]{equation}
%\begin{equation}
\Delta := \Delta_n := \menge{(x_1,\ldots,x_n)\in\RR^n_+}{x_1+x_2+\cdots +x_n=1},
%    \end{equation}
\end{empheq}
where $\RR_+ = \menge{x\in\RR}{x\geq 0}$.
The probability simplex is of central importance in
Statistics, Optimization, and Information Theory; see, e.g.,
\cite{CoverThomas} and \cite{Lange}.
(We write $\Delta_n$ if we wish to emphasize the dimension $n$.)
It will be convenient to set
\begin{empheq}[box=\mybluebox]{equation}
%\begin{equation}
\boldsymbol{1} := (1,1,\ldots,1)\in\RR^n
\quad\text{and}\quad
u := \big(\tfrac{1}{n},\tfrac{1}{n},\ldots,\tfrac{1}{n}\big)\in\Delta_n.
%    \end{equation}
\end{empheq}
% if we need to stress the dimension $n$, we will also write $\boldsymbol{1}_n$ instead of $\boldsymbol{1}$.
The problem we investigate is the following:
Given $p\in\Delta\smallsetminus\{u\}$,
there exist two unique points $a = a(p)$ and $b = b(p)$ in $\Delta$
such that
$\{u,p\}\subseteq [a,b]$ and $\|a-b\|$ is maximal.
(Let us note that when, e.g., $n=3$, then the mapping
$\Delta\smallsetminus\{u\}\colon\RR\colon p\mapsto\|a(p)-b(p)\|$ is continuous; however,
it is not possible to extend it continuously --- let alone in a smooth manner --- at the point $u$.)
The quantity $\|a-b\|$ can be thought of as the
``width`` of $\Delta$ with respect to $p$ and
\begin{equation}
\label{e:cosquot}
\cos\big(\measuredangle (a,b)\big) = \frac{\scal{a}{b}}{\|a\|\|b\|}
\end{equation}
as the cosine of the ``angle spread'' with respect to $p$.

The \emph{aim of this paper} is to \emph{minimize the angle spread},
which equivalently corresponds to \emph{maximizing its cosine}
\begin{equation}
\tag{P}
\label{e:P}
\max_{p\in\Delta\smallsetminus\{u\}}{\frac{\scal{a(p)}{b(p)}}{\|a(p)\|\|b(p)\|}}.
\end{equation}
% {\color{blue} Clearly, the objective function in \cref{e:P} is constant on $\Delta\smallsetminus\{u\}$ intersected with any
% line passing through $u$.}
When $n=2$, it is
clear that the maximum value of \cref{e:P} is $\cos(\pi/2)=0$
and that every $p\in\Delta\smallsetminus\{u\}$ solves \cref{e:P}.
Thus, we assume henceforth that
\begin{empheq}[box=\mybluebox]{equation}
n\in\{3,4,\ldots\}.
\end{empheq}
Our main result can now be stated:
\begin{theorem}
\label{t:main}
The maximum value of \cref{e:P}
is
\begin{equation}
\frac{n-2}{n+2}
\end{equation}
and a pair realizing this maximum is
$a^*=\tfrac{1}{n}(2,1,1,\ldots,1,0)$ and
$b^*=\tfrac{1}{n}(0,1,1,\ldots,1,2)$.
Consequently, the minimal angle spread is
\begin{equation}
\label{e:main}
\arccos\Big(\frac{n-2}{n+2}\Big).
\end{equation}
\end{theorem}
See \cref{figure} for an illustration of \cref{t:main} when $n=3$.
We note that as the dimension of the space $n$ increases to infinity,
the minimal angle spread approaches $\arccos(1)=0$.
We also note that when $n=1,2,4$, then
$\pi-\arccos((n-2)/(n+2))$ coincides with the dihedral angle of
the tetrahedron, cube, octahedron, respectively. However, dihedral
angles of other classical polyhedra (see \cite{dihedral}) do not seem to be
related to the angles provided by \cref{e:main}.

Our strategy to prove \cref{t:main} is to reduce the complexity of the problem
in stages by exploiting its structure.
Eventually, we are led to a one-dimensional problem
which we then solve by Calculus.

\begin{figure}
\centering
\includegraphics[scale = 0.3]{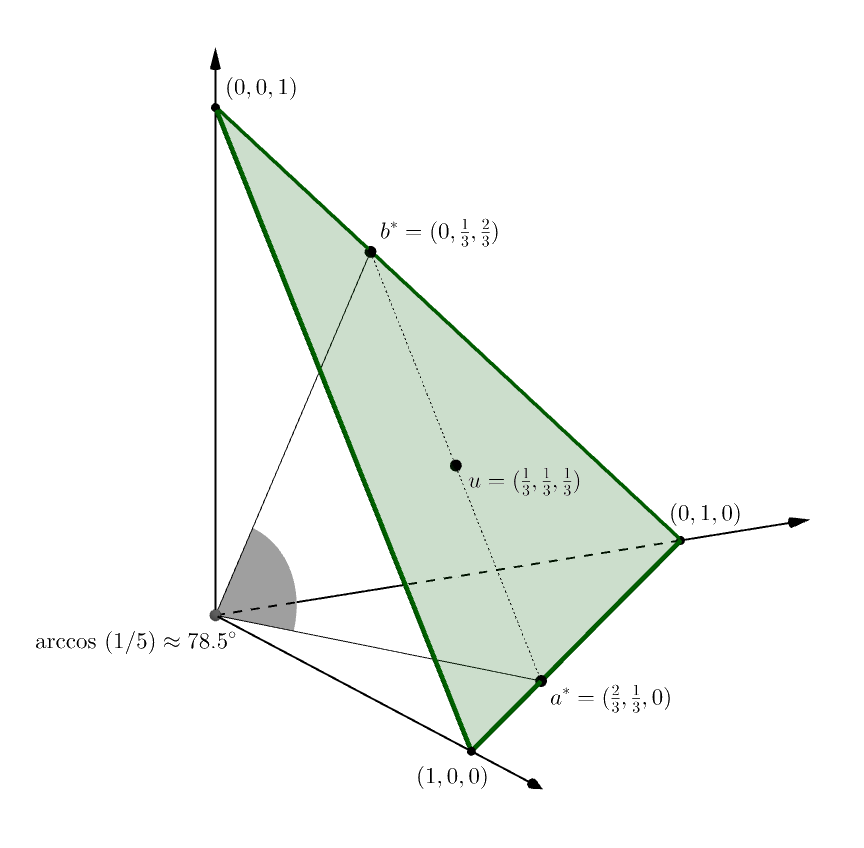}
\caption{An illustration of \cref{t:main} when $n=3$.} \label{figure}
\end{figure}

The remainder of this paper is organized as follows.
In \cref{sec:aux}, we collect a technical optimization result that will
be used later. The computation of $[a,b]$ is carried out in \cref{sec:ab}.
In \cref{sec:quot}, we set up the cosine quotient \cref{e:cosquot} in a
more tractable form. The optimization is then tackled in
\cref{sec:max1} where we keep $p_1$ and $p_n$ fixed.
At last, the proof is completed in \cref{sec:max2}.
In \cref{sec:app}, we sketch an application of our results
--- which in fact motivated this note --- in cognitive science.
The presented results are a first contribution to using the minimal
angle spread as a nonsmooth discrepancy term in this area.

Finally, we note that a reviewer pointed out that another possible departure point
is to study sufficient conditions for strong convexity of integral functionals
in two-stage stochastic linear programming. In particular, recent work by
Claus and Sp\"urkel \cite{ClSp} features minimal angles of similar though different kind.
This is a promising direction to explore in future research.

The notation in this paper is fairly standard and follows largely \cite{BC2017}.

\section{An auxiliary result}

\label{sec:aux}

It will be convenient to have the following result ready for future use.

\begin{lemma}
\label{l:210511a}
Let $m\in\{1,2,\ldots\}$, set $\boldsymbol{1} := (1,1,\ldots,1)\in\RR^m$, and let $\gamma\in\RR$.
Set $h\colon\RR^m\to\RP\colon x\mapsto \|x-\gamma \boldsymbol{1}\|^2$,
let $x\in\RR^m$, and set
\begin{equation*}
y := (\eta,\eta,\ldots,\eta)\in\RR^m,
\quad\text{where}\quad
\eta := \tfrac{1}{m}(x_1+x_2+\cdots+x_m).
\end{equation*}
Then
\begin{equation*}
h(y) \leq h(x).
\end{equation*}
\end{lemma}
\begin{proof}
Indeed, we have
\begin{align*}
h(y) - \gamma^2\|\boldsymbol{1}\|^2
&=
\|y\|^2 - 2\gamma\scal{y}{\boldsymbol{1}}
=
m\eta^2 - 2\gamma m\eta\\
&= m\big(\tfrac{1}{m}(x_1+x_2+\cdots+x_m) \big)^2
- 2\gamma(x_1+x_2+\cdots+x_m)\\
&\leq
m \tfrac{1}{m}\big(x_1^2+x_2^2+\cdots+x_m^2\big) - 2\gamma\scal{x}{\boldsymbol{1}}\\
&=\|x\|^2 - 2\gamma\scal{x}{\boldsymbol{1}}\\
&= h(x)- \gamma^2\|\boldsymbol{1}\|^2,
\end{align*}
where the inequality follows from  the convexity of the square function.
\end{proof}

\begin{corollary}
\label{c:210511b}
Let $m\in\{1,2,\ldots\}$,
and let $\gamma\in\RR$.
Set
$A\colon \RR^m\to\RR^m\colon x \mapsto (\eta,\eta,\ldots,\eta)$,
where $\eta = \tfrac{1}{m}(x_1+x_2+\cdots+x_m)$, and
let $C$ be a nonempty subset of $\RR^m$ such that
$A(C)\subseteq C$.
Set
$\boldsymbol{1} := (1,1,\ldots,1)\in\RR^m$.
Then to
\begin{equation}
\text{minimize}\;\;
\|x-\gamma\boldsymbol{1}\|^2
\;\;\text{over $x\in C$}
\end{equation}
is the same as to
\begin{equation}
\text{minimize}\;\;
\|x-\gamma\boldsymbol{1}\|^2
\;\;\text{over $x\in C\cap \RR\,\boldsymbol{1}$}
\end{equation}
in the sense that the optimal values for both problems
are identical and if the first problem has a solution, then
it also has a solution that solves also the second problem.
\end{corollary}
\begin{proof}
Clear from \cref{l:210511a}.
\end{proof}

\begin{remark}
We point out in passing that the operator $A$ from
\cref{c:210511b} is the projection operator of the set $\RR\boldsymbol{1}$.
\end{remark}

\section{Determining the relative boundary points $\scriptstyle a$ and $\scriptstyle b$}

\label{sec:ab}

In this section, we fix
\begin{equation}
\label{e:pass1}
p\in \Delta\smallsetminus\{u\}.
\end{equation}
Then there exist indices $i$ and $j$ in $I := \{1,2,\ldots,n\}$ such that
$p_i<\tfrac{1}{n}<p_j$.
Without loss of generality, we assume that
\begin{equation}
\label{e:pass2}
0 \leq p_1 = \min_{i\in I} p_i < \tfrac{1}{n} < p_n = \max_{i\in I} p_i \leq 1.
\end{equation}
Set
\begin{equation}
(\forall \lambda\in\RR)\quad
q(\lambda) := (1-\lambda)u+\lambda p = u + \lambda(p-u).
\end{equation}
Note that
$q(0)=u$, $q(1)=p$,
\begin{equation}
(\forall \lambda\in\RR)(\forall i\in I)\quad
q_i(\lambda) = (1-\lambda)u_i+\lambda p_i = u_i + \lambda(p_i-u_i),
\end{equation}
and hence
\begin{equation}
q_1(\lambda)+q_2(\lambda)+\cdots + q_n(\lambda) = 1.
\end{equation}
We wish to find the smallest and largest $\lambda\in\RR$ such that
$q(\lambda)\in\Delta$.
Let $i\in I$.
Suppose that $p_i > u_i$.
Then $(\forall \lambda \in\mathbb{R}_+)$ $q_i(\lambda)>u_i$.
So if $\lambda < 0$, then
$q_i(\lambda) = 0$
$\Leftrightarrow$
$-u_i = \lambda(p_i-u_i)$
$\Leftrightarrow$
$\lambda = -u_i/(p_i-u_i) = -(1/n)/(p_i-(1/n))=-1/(np_i-1)\leq -1/(np_n-1)$.
Hence the smallest value $\lambda$ still guaranteeing $q(\lambda)\in\Delta$ is
\begin{equation}
\lambda_- := \frac{-1}{np_n-1}= \frac{1}{1-np_n} < 0.
\end{equation}
Analogously, the largest value $\lambda$ still guaranteeing $q(\lambda)\in\Delta$ is
\begin{equation}
\lambda_+ := \frac{1}{1-np_1} > 0.
\end{equation}
The corresponding vectors
\begin{equation}
a := q(\lambda_-) = u + \lambda_-(p-u)
\end{equation}
and
\begin{equation}
b := q(\lambda_+) = u + \lambda_+(p-u)
\end{equation}
thus form the largest segment $[a,b]\subseteq \Delta$
such that $\{u,p\}\subseteq [a,b]$.
Note that $a$ and $b$ are depending on $p$ --- when we want
to stress this, then we'll write $a(p)$ and $b(p)$.

For $i\in I$, we simplify
\begin{equation}
\label{e:ai}
a_i
= u_i+\lambda_-(p_i-u_i)
= \tfrac{1}{n}+\frac{1}{1-np_n}\big(p_i-\tfrac{1}{n}\big)
= \frac{1-np_n+n\big(p_i-\tfrac{1}{n}\big)}{n(1-np_n)}
= \frac{p_n-p_i}{np_n-1}
\end{equation}
and similarly
\begin{equation}
\label{e:bi}
b_i = \frac{p_i-p_1}{1-np_1}.
\end{equation}
Thus
\begin{equation}
\label{e:220203a}
a = \frac{1}{np_n-1}(p_n\boldsymbol{1}-p)
\;\;\text{and}\;\;
b = \frac{1}{1-np_1}(p-p_1\boldsymbol{1}),
\end{equation}
where $\boldsymbol{1}=(1,1,\ldots,1)\in\RR^n$.
Next,
\begin{subequations}
\label{e:norma^2}
\begin{align}
\|a\|^2
&=
\frac{1}{(np_n-1)^2}\|p_n\boldsymbol{1}-p\|^2
=
\frac{1}{(np_n-1)^2}\big(p_n^2\|\boldsymbol{1}\|^2-
2p_n\scal{\boldsymbol{1}}{p}+\|p\|^2\big)\\
&=
\frac{1}{(np_n-1)^2}\big(np_n^2-2p_n+\|p\|^2\big)\\
&=
\frac{1}{(np_n-1)^2}\Big(p_1^2+\big(p_2^2+\cdots + p_{n-1}^2 \big)+(n+1)p_n^2-2p_n\Big)
\end{align}
\end{subequations}
and similarly
\begin{equation}
\label{e:normb^2}
\|b\|^2 =
\frac{1}{(1-np_1)^2}\Big((n+1)p_1^2-2p_1+\big(p_2^2+\cdots + p_{n-1}^2 \big)+p_n^2\Big).
\end{equation}

\section{Setting up the cosine quotient}
\label{sec:quot}

We uphold the assumptions and notation from the previous section.
It is convenient to abbreviate
\begin{equation}
z := (p_2,\ldots,p_{n-1})\in\RR^{n-2}.
\end{equation}
Note that
\begin{equation}
\label{e:preC}
(\forall i\in\{1,2,\ldots,n-2\})
\;\; p_1 \leq z_i \leq p_n
\quad\text{and}\quad
\sum_{i=1}^{n-2}z_i = 1- p_1-p_n \geq 0
\end{equation}
because of \cref{e:pass1} and \cref{e:pass2}.
This allows us to rewrite \cref{e:norma^2} and \cref{e:normb^2} more
succinctly as
\begin{equation}
\label{e:nicenorma^2}
\|a\|^2 = \frac{1}{(np_n-1)^2}\big(p_1^2+\|z\|^2 +(n+1)p_n^2-2p_n\big)
\end{equation}
and
\begin{equation}
\label{e:nicenormb^2}
\|b\|^2 =
\frac{1}{(1-np_1)^2}\big((n+1)p_1^2-2p_1+\|z\|^2+p_n^2\big).
\end{equation}
Next, using \cref{e:ai} and \cref{e:bi}, we have
\begin{subequations}
\label{e:niceab}
\begin{align}
\scal{a}{b}
&= \sum_{i=1}^{n} a_ib_i
=
\sum_{i=1}^n \frac{p_n-p_i}{np_n-1}\cdot\frac{p_i-p_1}{1-np_1}
\\
&=
\frac{1}{(np_n-1)(1-np_1)}
\sum_{i=2}^{n-1} (p_n-p_i)(p_i-p_1)
\\
&=
\frac{1}{(np_n-1)(1-np_1)}
\sum_{i=1}^{n-2} (p_n-z_i)(z_i-p_1).
\end{align}
\end{subequations}
Using \cref{e:niceab}, \cref{e:nicenorma^2}, and
\cref{e:nicenormb^2},
we now set up the quotient of interest from \cref{e:cosquot}:
\begin{equation}
\label{e:cosquot2}
\frac{\scal{a}{b}}{\|a\|\|b\|}
=
\frac{\displaystyle\sum_{i=1}^{n-2} (p_n-z_i)(z_i-p_1)}{\sqrt{p_1^2+\|z\|^2 +(n+1)p_n^2-2p_n}\sqrt{(n+1)p_1^2-2p_1+\|z\|^2+p_n^2}}.
\end{equation}

\section{Maximizing the cosine quotient (with $\scriptstyle p_1$ and $\scriptstyle p_n$ fixed)}

\label{sec:max1}

We uphold the notation of the previous section.
Now we turn toward maximizing the cosine quotient \cref{e:cosquot2},
with $p_1$ and $p_n$ fixed.
In view of \cref{e:preC}, $z=(p_2,\ldots,p_{n-1})$ must belong to the
compact convex set
\begin{equation}
C := \mmenge{w\in[p_1,p_n]^{n-2}}{\sum_{i=1}^{n-2}w_i = 1- p_1-p_n}.
\end{equation}
The continuity of \cref{e:cosquot2} as a function of $z$
coupled with the compactness of $C$ guarantees the existence of
a maximizer of the cosine quotient.
Note that
\begin{equation}
A(C)\subseteq C,
\end{equation}
where $A\colon\RR^{n-2}\to\RR^{n-2}\colon
w\mapsto (\eta,\eta,\ldots,\eta)\in\RR^{n-2}$
and $\eta = \tfrac{1}{n-2}(w_1+w_2+\cdots+w_{n-2})$.

In general, maximizing a quotient is more involved;
however, we will get a lucky break for our problem:
It turns out we can maximize the numerator and minimize the denominator
of \cref{e:cosquot2} with respect to $z$ and we luckily obtain
the same optimal vector!

For convenience, set
$\alpha := p_1$, $\beta := p_n$, $\gamma := (\alpha+\beta)/2$,
and $\boldsymbol{1} := (1,1,\ldots,1)\in \RR^{n-2}$.
Then the numerator of \cref{e:cosquot2} --- as a function of $z$ --- is
\begin{subequations}
\begin{align}
\sum_{i=1}^{n-2} (\beta-z_i)(z_i-\alpha)
&=
-\sum_{i=1}^{n-2} z_i^2 + (\alpha+\beta)\sum_{i=1}^{n-2} z_i
+ \text{constant}\\
&=
-\|z\|^2 + (\alpha+\beta)\scal{\boldsymbol{1}}{z}
+ \text{constant}\\
&= -\|z-\gamma\boldsymbol{1}\|^2
+ \text{constant}.
\end{align}
\end{subequations}
We want to maximize the numerator over $z\in C$;
equivalently, we want to minimize $\|z-\gamma\boldsymbol{1}\|^2$
over $z\in C$.
By \cref{c:210511b}, we may restrict our attention
to
\begin{equation}
\label{e:210512c}
z \in C\cap \RR\boldsymbol{1}.
\end{equation}
On to the denominator of \cref{e:cosquot2}!
It is clear that the denominator becomes small when $\|z\|$ becomes small.
So minimizing the denominator corresponds to
minimizing $\|z\|^2 = \|z-0\boldsymbol{1}\|^2$.
Again by \cref{c:210511b}, we may restrict our attention
to \cref{e:210512c}!
But if $z\in C\cap\RR\boldsymbol{1}$, say $z=\zeta\boldsymbol{1}$
for $\zeta\in\RR$, then
the requirement that $\zeta\boldsymbol{1}\in C$ forces
$\sum_{i=1}^{n-2}z_i = (n-2)\zeta = 1-p_1-p_n$, i.e.,
$\zeta = (1-p_1-p_n)/(n-2)$.
Because $(n-1)p_1+p_n\leq 1 \leq p_1+(n-1)p_n$, it follows that
$p_1\leq \zeta \leq p_n$.
Altogether,
\begin{equation}
\label{e:210512z}
z = \frac{1-p_1-p_n}{n-2}(1,1,\ldots,1)\in C\cap\RR\boldsymbol{1} \subseteq \RR^{n-2}
\end{equation} maximizes \cref{e:cosquot2}
and
\begin{equation}
\label{e:210512z^2}
\|z\|^2 = \frac{(1-p_1-p_n)^2}{n-2}.
\end{equation}
To sum up, when restricted to the one-dimensional slice $C\cap\RR\boldsymbol{1}$ (see \cref{e:210512c}),
we obtain the unique solution  $z$ given by \cref{e:210512z}.
Our next step is to plug \cref{e:210512z} and \cref{e:210512z^2}
back into \cref{e:cosquot2}.
We start with the numerator of \cref{e:cosquot2}:
\begin{subequations}
\label{e:pfizer1}
\begin{align}
\scal{a}{b}
&=
\sum_{i=1}^{n-2}
\Big(p_n-\frac{1-p_1-p_n}{n-2}\Big)\Big(\frac{1-p_1-p_n}{n-2} -p_1\Big)\\
&=
(n-2)\Big(\frac{(n-2)p_n - (1-p_1-p_n)}{n-2}\Big)
\Big(\frac{1-p_1-p_n-(n-2)p_1}{n-2}\Big)\\
&=
\frac{1}{n-2}\big(p_1+(n-1)p_n-1\big)\big(1-p_n-(n-1)p_1\big).
\end{align}
\end{subequations}
Using \cref{e:210512z^2}, we see that the square of the denominator of \cref{e:cosquot2}
is
\begin{subequations}
\label{e:pfizer2}
\begin{align}
\|a\|^2\|b\|^2
&= \bigg(p_1^2 + \frac{(1-p_1-p_n)^2}{n-2} + (n+1)p_n^2 - 2p_n\bigg)\\
&\quad \times \bigg((n+1)p_1^2 - 2p_1 + \frac{(1-p_1-p_n)^2}{n-2} + p_n^2 \bigg)\\
&= \frac{1}{(n-2)^2}\\
&\quad \times \big((n-2)p_1^2 + (1-p_1-p_n)^2 + (n-2)(n+1)p_n^2 - 2(n-2)p_n \big)\\
& \quad \times \big((n-2)(n+1)p_1^2 - 2(n-2)p_1 + (1-p_1-p_n)^2 +
(n-2)p_n^2 \big).
\end{align}
\end{subequations}
Abbreviating $x=p_1$ and $y=p_n$,
we use \cref{e:pfizer1} and \cref{e:pfizer2} to write the square of \cref{e:cosquot2} as
\label{e:Mess}
\begin{align}
\frac{\scal{a}{b}^2}{\|a\|^2\|b\|^2}
&= \frac{\frac{\big(x+(n-1)y-1\big)^2\big(1-y-(n-1)x\big)^2}
{\big((n-2)x^2 + (1-x-y)^2 + (n-2)(n+1)y^2-2(n-2)y \big)}}
{\big((n-2)(n+1)x^2 -2(n-2)x + (1-x-y)^2 + (n-2)y^2 \big)}.
\end{align}

In the next section, we will do the final maximization by
setting $p_1$ and $p_n$, i.e., $x$ and $y$, loose.

\section{Maximizing the cosine quotient (concluded)}

\label{sec:max2}

While \cref{e:Mess} looks like a huge mess, we can invoke one small
but useful optimization: the objective function value of the
cosine quotient is by construction constant on $[a,b]$.
Now both $a$ and $b$ have at least one coordinate equal to $0$,
namely $a_n=0$ and $b_1=0$ (see \cref{e:220203a}).
Because the objective function in \cref{e:P} is constant on
$\Delta\smallsetminus\{u\}$ intersected with any line passing
through $u$,
we may and do finally assume that $a=p$ and thus $p_1=x=0$.
Then \cref{e:Mess} simplifies to
%\begin{subequations}
\begin{equation}
\label{e:mess}
\frac{\scal{a}{b}^2}{\|a\|^2\|b\|^2}
= \frac{\big((n-1)y-1\big)^2\big(1-y\big)^2}
{\big((1-y)^2 + (n-2)(n+1)y^2-2(n-2)y \big)
{\big((1-y)^2 + (n-2)y^2 \big)}
}
=: Q(y).
\end{equation}
%\end{subequations}
The constraints on $p$ are
$0 \leq p_1 \leq p_2 = \cdots = p_{n-1} = (1-p_1-p_n)/(n-2) \leq p_n\leq 1$.
With our assumption that
$p_1=0=x$ and $p_n=y$, these simplify to
$0 \leq (1-y)/(n-2)\leq y \leq 1$ and then to
\begin{equation}
\label{e:ycons}
\frac{1}{n-1}\leq y\leq 1.
\end{equation}
Because $n\geq 3>2$, we have
\begin{equation}
\frac{1}{n-1}< \frac{2}{n} < 1.
\end{equation}
So our remaining goal is to
\begin{equation}
\label{e:lastmax}
\text{maximize\;\;}
Q(y)
\;\;\text{subject to}
\;\;
\frac{1}{n-1}\leq y\leq 1
\end{equation}
where $Q(y)$ is defined in \cref{e:mess}
Using the chain quotient rule to compute
the derivative of $Q(y)$ followed by
factoring yields
\begin{align}
Q'(y)
&=
\frac{2(n-2)\big((n-1)^2y^2-ny+1\big)\big((n-1)y-1\big)(ny-2)(y-1)y}
{\big((n^2-n-1)y^2 + 2(1-n)y +1\big)^2 \big((n-1)y^2 -2y+1 \big)^2}.
\end{align}
Note that $Q'(y)$ has six roots, namely
\begin{equation}
\bigg\{\frac{n\pm\sqrt{-(3n-2)(n-2)}}{2(n-1)^2}, 0, \frac{1}{n-1}, \frac{2}{n},1\bigg\}.
\end{equation}
Because $n>2$ implies $-(3n-2)(n-2)<0$,
we note that there are exactly four real roots.
Our constraint \cref{e:ycons} excludes $0$.
The remaining three roots include the endpoints of
interval constraint.
Thus, the maximum value of \cref{e:lastmax} is found
by substituting for $y$ the values ${1}/(n-1),{2}/{n},1$ into
$Q(y)$ which results in $0,(n-2)^2/(n+2)^2,0$.

To sum up, $y=\frac{2}{n}$ is the solution of \cref{e:lastmax},
with maximum value $(n-2)^2/(n+2)^2$.
Using \cref{e:210512z}, we see this gives rise
to the probability distribution
\begin{equation}
\label{e:p*}
p^* = \tfrac{1}{n}\big(0,1,1,\ldots,1,2\big).
\end{equation}
We have shown that the angle spread is minimized
for $p^*$ given by \cref{e:p*}.
In view of \cref{e:ai} and \cref{e:bi},
this gives rise to
\begin{equation}
a^* := a(p^*) = \tfrac{1}{n}\big(2,1,1,\ldots,1,0\big)
\;\;\text{and}\;\;
b^* := b(p^*) = \tfrac{1}{n}\big(0,1,1,\ldots,1,2\big) = p^*.
\end{equation}
These vectors give the optimal value
\begin{equation}
\label{e:theend}
\cos\big(\measuredangle (a^*,b^*)\big) = \frac{\scal{a^*}{b^*}}
{\|a^*\|\|b^*\|} = \frac{n-2}{n+2}
\end{equation}
which we claimed in the introduction.
The proof of \cref{t:main} is thus complete.

\begin{remark}
We point out that the expression for the optimal value
in \cref{e:theend} was discovered numerically using
{\texttt{Julia}} \cite{Julia}. The computations in this section were verified with
{\texttt{SageMath}} \cite{Sage}. Finally, \cref{figure} was created using
{\texttt{Geogebra}} \cite{Geogebra}.
\end{remark}

\section{An application in cognitive science}

\label{sec:app}

Our main result (\cref{t:main}) provides an important theoretical and methodological advancement in the study of probabilistic belief models in humans. Human learning is often modeled with Bayesian statistics, where beliefs are represented as probability distributions over possible outcomes \cite{tradeoff, optimal, approxbayes, growmind}.

\subsection*{A motivating example}
Consider the expected returns on a \$100 investment over 1 year. A ``bull'' investor believes market prices will rise, while a ``bear'' investor believes market prices will fall. However, a bull still surely accepts prices \emph{could} fall, despite holding this belief to a lesser degree than the bear. In Bayesian terms, the beliefs of our two investors can be represented as probability distributions over the domain of potential investment returns. The bull's ``prior'' (initial belief) will have most of its probability mass over positive returns, while the opposite is true for the bear's prior.
These priors change over time in light of new evidence. While it is theoretically optimal to update according to the laws of probability, the empirical question remains as to how people actually represent and update probabilistic beliefs.

\subsection*{Previous work}
Much of the experimental work in this field infers prior and posterior beliefs from sequential participant actions, or elicits them with insufficient detail \cite{elicit, statelicit, expectations}. This requires experimenters to make unwarranted assumptions about the way that probabilistic beliefs are parameterized as probability distributions, and the nature of how they are updated with new data. It may be that a strict Bayesian model of human cognition is inappropriate when these assumptions are relaxed. Recent work by DiBerardino, Filipowicz, Danckert, and Anderson
used a computerized version of the game ``Plinko'', where at each trial a ball falls through an array of pegs into one of $n=40$ slots below.
Participants were tasked with estimating the distribution of future ball drops by explicitly drawing a probability distribution that represents their beliefs after each trial \cite{plinko}. This work found that initial beliefs (priors) vary across individuals and that these differences indicated future ball drop learning accuracy, thus demonstrating the importance of directly measuring individual beliefs in a theoretically agnostic manner \cite{plinko}.

However, \cite{plinko} cannot make any additional claims about \emph{why} participants with some priors appear to learn better than others. Its main limitation is that some participants' priors are more similar to the forthcoming ball drop distribution than others. As a result, it cannot be determined whether or not the observed differences in learning ability across different priors are due to participant ``state'' or ``trait'' differences. It may be that some participants are better probabilistic learners, which is a trait that can be detected by the shape of their prior beliefs. But it may also be that the participants who performed the best only did so because they just happened to have an initial state of belief more amenable to learning the forthcoming distribution.

\subsection*{On-going work}

Our main result (\cref{t:main}) allows for a new Plinko experiment that presents each participant with a ball drop distribution of a set level of similarity to any arbitrary participant's prior. Specifically, when measuring similarity between discrete probability distributions as the angular spread between representative Euclidean vectors, as was done in \cite{plinko}, we can determine the maximum possible rotation (dissimilarity) for any given vector (participant's prior) with respect to the uniform distribution. We require the following result. %This rotated vector now defines the probability distribution that participants will be tasked to learn.

\begin{corollary}
\label{c:app}
For every $p\in\Delta_n\smallsetminus\{u\}$
there exist $q$ and $v$ in $\Delta_n \cap (u+\RR(p-u))$
such that
\begin{equation}
\measuredangle (p,q) = \frac{1}{2}\arccos\Big(\frac{n-2}{n+2}\Big)
= \measuredangle(u,v).
\end{equation}
\end{corollary}
\begin{proof}
Let $p\in\Delta_n\smallsetminus\{u\}$, and set
$\alpha_n := \arccos((n-2)/(n+2))$.
We have seen earlier (see \cref{sec:ab}) that
there exist two unique points $a = a(p)$ and $b = b(p)$ in $\Delta_n$
such that $\{u,p\}\subseteq [a,b]$ and $\|a-b\|$ is maximal.
Consequently,
\begin{equation}
\measuredangle (a,p) + \measuredangle (p,b) = \measuredangle (a,b)
= \measuredangle(a,u) + \measuredangle(u,b).
\end{equation}
On the other hand,
\cref{t:main} yields $\measuredangle(a,b)\geq \alpha_n$.
Altogether,
\begin{equation}
\measuredangle (a,p) + \measuredangle (p,b) =
\measuredangle(a,u) + \measuredangle(u,b)
\geq \alpha_n = \arccos\Big(\frac{n-2}{n+2}\Big).
\end{equation}
Thus if $\measuredangle (a,p) \geq \alpha_n/2$,
then rotating $p$ towards $a$ yields $q$.
If
$\measuredangle (a,p) < \alpha_n/2$, then
$\measuredangle (p,b) >  \alpha_n/2$ and we rotate $p$ towards $b$
to obtain $q$.
Similarly,
if $\measuredangle (a,u) \geq \alpha_n/2$,
then rotating $u$ towards $a$ yields $v$.
Finally, if
$\measuredangle (a,u) < \alpha_n/2$, then
$\measuredangle (u,b) >  \alpha_n/2$ and we rotate $u$ towards $b$
to obtain $v$.
\end{proof}

By \cref{c:app}, if a participant produces a prior
$p\in\Delta_n\smallsetminus\{u\}$, then the experimenter presents a ball drop
distribution $q$ where $\measuredangle (p,q) = \alpha_n/2$. If a participant
produces the uniform prior $u$, then the experimenter presents a ball drop
distribution $v$ where $\measuredangle (u,v) = \alpha_n/2$ arising from an
arbitrary $p\in \Delta_n\smallsetminus\{u\}$ via  \cref{c:app}.

More psychological research is required to determine the precise measure of
similarity humans use to compare probability distributions. Angular similarity
is only one such possibility. This result also poses an interesting theoretical
development: Does the shrinking minimal angle spread as the dimension tends to
infinity create any difficulties for people to perceive or act upon any
particular set of probability distributions? It may be that thinking about
opposite/dissimilar probability distributions becomes more difficult as the
number of discrete histogram bins approaches infinity if these mathematics
correspond to mechanisms of human inference.

% HIER

\nocite{*}
%%%%%%%%%%%%%%%%%%%%%%%%%%%%%%%%%%%%%%%%%%%%%%%
\bibliographystyle{jnsao}

%\bibliography{jnsao_template}
%%%%%%%%%%%%%%%%%%%%%%%%%%%%%%%%%%%%%%%%%%%%%%%
%%%%%%%%%%%%%%%%%%%%%%%%%%%%%%%%%%%%%%%%%%%%%%%
\end{document}